\documentclass[11pt]{amsart}

\usepackage{comment}

\newtheorem{theorem}{Theorem}[section]

\newtheorem{corollary}[theorem]{Corollary}

\newtheorem{lemma}[theorem]{Lemma}

\newtheorem{proposition}[theorem]{Proposition}

\theoremstyle{definition}
\newtheorem{remark}[theorem]{Remark}

\newtheorem{example}[theorem]{Example}

\usepackage[colorlinks, bookmarks=true]{hyperref}
\usepackage{color,graphicx,shortvrb}

\usepackage{enumerate}
\usepackage{amsmath}
\usepackage{amssymb}

\def\query#1{\setlength\marginparwidth{80pt}%
\marginpar{\raggedright\fontsize{10}{10}\selectfont\itshape
\hrule\smallskip
{\textcolor{red}{#1}}\par\smallskip\hrule}}

\newcommand{\one}{\mathbf{1}}

\newcommand{\ran}{\mathrm{ran} \,}

\newcommand{\spn}{\mathrm{span}}
\newcommand{\vr}{\varepsilon}

\newcommand{\ball}{\mathbf{B}}
\newcommand{\sphere}{\mathbf{S}}
\newcommand{\ep}{\mathrm{EP}}
\newcommand{\oep}{\mathrm{OEP}}

\newcommand{\N}{\mathbb{N}}

\newcommand{\R}{\mathbb{R}}

\newcommand{\csch}{{\mathrm{CSCH}}}
\newcommand{\sch}{{\mathrm{SCH}}}
\newcommand{\ch}{{\mathrm{CH}}}
\newcommand{\so}{{\mathrm{S}}}

\newcommand{\triple}[1]{{\left\vert\kern-0.25ex\left\vert\kern-0.25ex\left\vert #1 
    \right\vert\kern-0.25ex\right\vert\kern-0.25ex\right\vert}}

\begin{document}

\numberwithin{equation}{section}


\title[Order extreme points]{Order extreme points and solid convex hulls}

\author{T.~Oikhberg and M.A.~Tursi}

\address{ Dept.~of Mathematics, University of Illinois, Urbana IL 61801, USA}
\email{oikhberg@illinois.edu, gramcko2@illinois.edu}

\date{\today}



\dedicatory{To the memory of Victor Lomonosov}


\maketitle

\parindent=0pt
\parskip=3pt

\begin{abstract}
We consider the ``order'' analogues of some classical notions of Banach space geometry: extreme points and convex hulls. A Hahn-Banach type separation result is obtained, which allows us to establish an ``order'' Krein-Milman Theorem. We show that the unit ball of any infinite dimensional reflexive space contains uncountably many order extreme points, and investigate the set of positive norm-attaining functionals. Finally, we introduce the ``solid'' version of the Krein-Milman Property, and show it is equivalent to the Radon-Nikod{\'y}m Property.
\end{abstract}

\maketitle
\thispagestyle{empty}

\section{Introduction}\label{s:intro}

At the very heart of Banach space geometry lies the study of three interrelated subjects: (i) separation results (starting from the Hahn-Banach Theorem), (ii) the structure of extreme points, and (iii) convex hulls (for instance, the Krein-Milman Theorem on convex hulls of extreme points).
Certain counterparts of these notions exist in the theory of Banach lattices as well.
For instance, there are positive separation/extension results; see e.g. \cite[Section 1.2]{AB}.
One can view solid convex hulls as lattice analogues of convex hulls; these objects have been studied, and we mention some of their properties in the paper. However, 
no unified treatment of all three phenomena listed above has been attempted.

In the present paper, we endeavor to investigate the lattice versions of (i), (ii), and (iii) above.
We introduce the order version of the classical notion of an extreme point: if $A$ is a subset of a Banach lattice $X$, then $a \in A$ is called an \emph{order extreme point} of $A$ if for all $x_0, x_1 \in A$ and $t \in (0,1)$ the inequality $a \leq (1-t) x_0 + t x_1$ implies $x_0 = a = x_1$.
Note that, in this case, if $x \geq a$ and $x \in A$, then $x = a$ (write $a \leq (x+a)/2$).
%

Throughout, we work with real spaces. 
We will be using the standard Banach lattice results and terminology (found in, for instance, \cite{AB}, \cite{M-N} or \cite{Sch}).
We also say that a subset of a Banach lattice is \textit{bounded} when it is norm bounded, as opposed to order bounded.

Some special notation is introduced in Section \ref{s:definitions}. In the same section, we establish some basic facts about order extreme points and solid hulls.
In particular, we note a connection between order and ``canonical'' extreme points (Theorem \ref{t:connection}).

In Section \ref{s:separation} we prove a ``Hahn-Banach'' type result (Proposition \ref{p:separation2}), involving separation by positive functionals. This result is used in Section \ref{s:KM} to establish a ``solid'' analogue of the Krein-Milman Theorem.
We prove that solid compact sets are solid convex hulls of their order extreme points (see Theorem \ref{t:KM}).
A ``solid'' Milman Theorem is also proved (Theorem \ref{t:order-milman}).


In Section \ref{s:examples} we study order extreme points in $AM$-spaces. For instance, we show that, for an AM-space $X$, the following three statements are equivalent: (i) $X$ is a $C(K)$ space; (ii) the unit ball of $X$ is the solid convex hull of finitely many of its elements; (iii) the unit ball of $X$ has an order extreme point (Propositions \ref{p:describe_AM} and \ref{p:only_C(K)_has_OEP}).

Further in Section \ref{s:examples} we investigate norm-attaining positive functionals.
Functionals attaining their maximum on certain sets have been investigated since the early days of functional analysis; here we must mention V.~Lomonosov's papers on the subject (see e.g.~the excellent summary \cite{ArL},
and the references contained there).
In this paper, we show that a separable AM-space is a $C(K)$ space iff any positive functional on it attains its norm (Proposition \ref{p:only_C(K)_attain_norm}). On the other hand, an order continuous lattice is reflexive iff every positive operator on it attains its norm (Proposition \ref{p:functional_OC}).

In Section \ref{s:how_many_extreme_points} we show that the unit ball of any
reflexive infinite-dimensional Banach lattice has uncountably many order extreme points
(Theorem \ref{t:uncount_many}).

Finally, in Section \ref{s:SKMP} we define the ``solid'' version of the Krein-Milman Property, and show that it is equivalent to the Radon-Nikodym Property (Theorem \ref{t:RNP}).

To close this introduction, we would like to mention that related ideas have been explored before, in other branches of functional analysis.
In the theory of $C^*$ algebras, and, later, operator spaces, the notions of ``matrix'' or ``$C^*$'' extreme points and convex hulls have been used. The reader is referred to e.g.~\cite{EWe}, \cite{EWi}, \cite{FaM}, \cite{WeWi} for more information; for a recent operator-valued separation theorem, see \cite{Mag}.

\section{Preliminaries}\label{s:definitions}

In this section, we introduce the notation commonly used in the paper, and mention some basic facts.

The closed unit ball (sphere) of a Banach space $X$ is denoted by $\ball(X)$ (resp.~$\sphere(X)$).
%
If $X$ is a Banach lattice, and $C \subset X$, write $C_+ = C \cap X_+$, where $X_+$ stands for the positive cone of $X$.
Further, we say that $C \subset X$ is \emph{solid} if,
for $x \in X$ and $z \in C$, the inequality $|x| \leq |z|$ implies the inclusion $x \in C$.
In particular, $x \in X$ belongs to $C$ if and only if $|x|$ does.
Note that any solid set is automatically \emph{balanced}; that is, $C = -C$.

Restricting our attention to the positive cone $X_+$, we say that $C \subset X_+$ is \emph{positive-solid} if for any $x\in X_+$, the existence of $z \in C$ satisfying $x \leq z$ implies the inclusion $x \in C$.

We will denote the set of order extreme points of $C$ (defined in Section \ref{s:intro}) by $\oep(C)$;
the set of ``classical'' extreme points is denoted by $\ep(C)$.

\begin{remark}\label{r:G_delta}
It is easy to see that the set of all extreme points of a compact metrizable set is $G_\delta$. The same can be said for the set of order extreme points of $A$, whenever $A$ is a closed  solid bounded subset of a separable reflexive Banach lattice.
Indeed, then the weak topology is induced by a metric $d$. For each $n$ let $F_n$ be the set of all $x \in A$ for which there exist $x_1, x_2, \in A$ with $x \leq (x_1 + x_2)/2$, and $d(x_1, x_2) \geq 1/n$. By compactness, $F_n$ is closed. Now observe that $\cup_n F_n$ is the complement of the set of all order extreme points.
\end{remark}

Note that every order extreme point is an extreme point in the usual sense, but the converse is not true:
for instance, $\one_{(0,1)}$ is an extreme point of $\ball(L_\infty(0,2))_+$, but not its order extreme point. However, a connection between ``classical'' and order extreme points exists:

\begin{theorem}\label{t:connection}
Suppose $A$ is a solid subset of a Banach lattice $X$.
Then $a$ is an extreme point of $A$ if and only if $|a|$ is its order extreme point.
\end{theorem}

The proof of Theorem \ref{t:connection} uses the notion of a quasi-unit.
Recall \cite[Definition 1.2.6]{M-N} that for $e,v \in X_+$, $v$ is a \textit{quasi-unit} of $e$ if $v \wedge (e-v) = 0$.  
This terminology is not universally accepted: the same objects can be referred to as \textit{components} \cite{AB}, or \textit{fragments} \cite{PR}.

\begin{proof}
Suppose $|a|$ is order extreme.  Let $0<t<1$ be such that $a = tx+(1-t)y$.  Then since $A$ is solid and  $|a| \leq t|x| + (1-t)|y|$, one has $|x| = |y| = |a|$. Thus the latter inequality is in fact equality. Thus $|a|+a = 2a_+ = 2t x_+ +2(1-t)y_+$, so $a_+ = tx_+ +(1-t)y_+$.  Similarly, $a_- = tx_- + (1-t)y_-$.  It follows that $x_+ \perp y_-$ and $x_- \perp y_+$.  Since $ x_+ + x_- =|x| = |y| = y_+ +y_-  $, we have that $x_+ = x_+ \wedge (y_+ +y_-) = x_+\wedge y_+ + x_+ \wedge y_-$ (since $y_+, y_-$ are disjoint).  Now since $x_+ \perp y_-$, the latter is just $x_+ \wedge y_+$, hence $x_+ \leq y_+$.  By similar argument one can show the opposite inequality to conclude that $x_+ = y_+$, and likewise $x_-=y_-$, so $x=y=a$.

Now suppose $a$ is extreme. It is sufficient to show that $|a| $ is order extreme for $A_+$.  Indeed, if $|a| \leq tx + (1-t)y$ (with $0 \leq t \leq 1$ and $x, y \in A$), then $|a| \leq t|x|+(1-t)|y|$. As $|a|$ is an order extreme point of $A_+$, we conclude  that $|x| = |y| = |a|$, so $|a| = tx+(1-t)y= t|x|+(1-t)|y|$.  The latter implies that $x_- = y_-=0$, hence $x=|x| =|a|=|y|=y$. 

Therefore, suppose $|a| \leq tx + (1-t)y$ with $0 \leq t \leq 1$, and $x,y \in A_+$. First show that $|a|$ is a quasi-unit of $x$ (and by similar argument of $y$). To this end, note that $a_+ - tx \wedge a_+ \leq (1-t)y\wedge a_+$.  Since $A$ is solid, 
\[A \ni z_+:= \frac{1}{1-t}( a_+ -tx\wedge a_+ ) \]
and similarly, since  $a_- - tx \wedge a_- \leq (1-t)y\wedge a_-$, 
\[A \ni z_-:= \frac{1}{1-t}( a_- -tx\wedge a_-) \]

These inequalities imply that $z_+ \perp z_-$, so they correspond to the positive and negative parts of some $z = z_+ -z_-$.  Also, $z\in A$ since $|z| \leq |a|$. Now $a_+ = t(x\wedge \frac{a_+}{t}) +(1-t) z_+$ and $a_- = t(x\wedge \frac{a_-}{t}) +(1-t) z_+$.  In addition, $|x\wedge\frac{a_+}{t} - x\wedge\frac{a_-}{t}| \leq x$, so since $A$ is solid,
\[z':= x\wedge\frac{a_+}{t} - x\wedge\frac{a_-}{t} \in A. \]
Therefore   $ a=a_+ -a_- = tz'+(1-t)z$.  Since $a$ is an extreme point, $a=z$, hence \[(1-t)z_+ =(1-t) a_+ =   a_+ -tx\wedge a_+ \] so $tx \wedge a_+ = ta_+ $ which implies that $(t(x-a_+))\wedge((1-t)a_+) = 0$. As $0<t<1$, we have that $a_+$ (and likewise $a_-$) is a quasi-unit of $x$ (and similarly of $y$).  Thus $|a|$ is a quasi-unit of $x$ and of $y$. 

Now let $s = x-|a|$.  Then $a+s, a-s \in A$, since $|a \pm s|= x$.  We have \[ a = \frac{a-s}{2} +\frac{a+s}{2}, \]  but since $a$ is extreme, $s$ must be $0$.  Hence $x=|a|$, and similarly $y=|a|$.
\end{proof}

The situation is different if $A$ is a positive-solid set: the paragraph preceding Theorem \ref{t:connection} shows that $A$ can have extreme points which are not order extreme. If, however, a positive-solid set satisfies certain compactness conditions, then some connections between extreme and order extreme points can be established; see 
Proposition \ref{p:under_oep}, and the remark following it.

If $C$ is a subset of a Banach lattice $X$, denote by $\so(C)$ the \emph{solid hull} of $C$, which is the smallest solid set containing $C$.
It is easy to see that $\so(C)$ is the set of all $z \in X$ for which there exists $x \in C$ satisfying $|z| \leq |x|$.
Clearly $\so(C) = \so(|C|)$, where $|C| = \{|x| : x \in C\}$.
Further, we denote by $\ch(C)$ the \emph{convex hull} of $C$. For future reference, observe:

\begin{proposition}\label{p:interchange}
If $X$ is a Banach lattice, then $\so(\ch(|C|)) = \ch(\so(C))$ for any $C \subset X$.
\end{proposition}

\begin{proof}
Let $x\in \ch(\so(C))$.  Then $x = \sum a_iy_i,$ where $\sum a_i = 1, a_i > 0$, and $|y_i| \leq |k_i|$ for some $k_i \in C$.  Then \[|x| \leq \sum a_i|y_i| \leq \sum a_i |k_i| \in \ch(|C|), \] so $x\in \so(\ch(|C|)).$ If $x\in \so(\ch(|C|))$, then 
\[|x| \leq \sum_1^n a_i y_i,\quad y_i \in |C|,\quad 0 < a_i, \quad \sum a_i = 1. \] 

We use induction on $n$ to prove that $x \in \ch(\so(C))$.
If $n= 1$, $x\in \so(C)$ and we are done.  Now, suppose we have shown that if $|x| \leq \sum_1^{n-1} a_iy_i$ then there are $z_1,...,z_{n-1} \in \so(C)_+$ such that $|x|= \sum_1^{n-1}a_iz_i$.  
From there, we have that 
\[|x| = (\sum_1^n a_i y_i)\wedge |x| \leq (\sum_1^{n-1} a_iy_i)\wedge |x| + (a_ny_n)\wedge |x|. \]
Now \[ 0 \leq |x| - (\sum_1^{n-1} a_iy_i)\wedge |x| \leq a_n(y_n\wedge \frac{|x|}{a_n}). \] 

Let $z_n :=\frac{1}{a_n}(|x| - (\sum_1^{n-1} a_iy_i)\wedge |x|)$.  By the above, $z_n \in \so(C)_+$. Furthermore,
\[\frac{1}{1-a_n}(|x| \wedge \sum_1^{n-1}a_iy_i) \leq \sum_1^{n-1}  \frac{a_i}{1-a_n} y_i \in \ch(|C|), \]
so by induction there exist $z_1,..,z_{n-1} \in \so(C)_+$ such that 
\[ |x|\wedge( \sum_1^{n-1}a_iy_i) = \sum_1^{n-1}  \frac{a_i}{1-a_n} z_i \]
Therefore $|x| = \sum_1^n a_iz_i$.   Now for each $n$, $a_iz_i \leq |x|$, so
$|x| = \sum \big( (a_iz_i) \wedge|x|\big)$, and 
\[a_iz_i = a_iz_i\wedge x_+ +a_iz_i\wedge x_- = a_i( z_i\wedge(\frac{x_+}{a_i}) + z_i\wedge(\frac{x_-}{a_i}) ).\]

Let $w_i =  z_i\wedge(\frac{x_+}{a_i}) - z_i\wedge(\frac{x_-}{a_i})$.  Note that $|w_i| = z_i$, so $w_i \in \so(C)$.  It follows that $x= \sum a_iw_i \in \ch(\so(C))$.
\end{proof}

For $C \subset X$ (as before, $X$ is a Banach lattice) we define 
the \emph{solid convex hull} of $C$ to be the smallest convex, solid set containing $C$, and denote it by $\sch(C)$; the norm (equivalently, weak) closure of the latter set is denoted by $\csch(C)$, and referred to as the \emph{closed solid convex hull} of $C$.

\begin{corollary}\label{c:equal-sch}
	Let $C\subseteq X$.  Then\begin{enumerate}
		\item  $\sch(C) = \ch(\so(C)) = \sch(|C|)$, and consequently, $\csch(C) = \csch(|C|)$.  
		\item If $C \subseteq X_+$, then $\sch(C) = \so(\ch(C))$.
	\end{enumerate}
\end{corollary}

\begin{proof}  
(1) Suppose $C\subseteq D$, where $D$ is convex and solid.  Then $\ch(\so(C)) \subseteq D$.
Consequently, $\ch(\so(C)) \subset \sch(C)$.
On the other hand, by Proposition \ref{p:interchange}, $\ch(\so(C))$ is also solid, so $\sch(C) \subseteq \ch(\so(C))$. Thus, $\sch(C) = \ch(\so(C)) = \ch(\so(|C|)) = \sch(|C|)$. \\
(2) This follows from (1) and the equality in Proposition \ref{p:interchange}.
%
\end{proof}

\begin{remark}\label{r:shadow-not-closed}
The two examples below show that $\so(C)$ need not be closed, even if $C$ itself is.
Example (1) exhibits an unbounded closed set $C$ with $\so(C)$ not closed; in Example (2), $C$ is closed and bounded, but the ambient Banach lattice needs to be infinite dimensional.

(1) Let $X$ be a Banach lattice of dimension at least two, and consider disjoint norm one $e_1, e_2 \in \ball(X)_+$.  Let $C = \{ x_n : n \in \N\}$, where $x_n = \frac{n}{n+1}e_1 +ne_2$. Now, $C$ is norm-closed: if $m > n$, then $\|x_m-x_n\| \geq \|e_2\| = 1$. 
However, $\so(C)$ is not closed: it contains $r e_1$ for any $r \in (0,1)$, but not $e_1$.

(2) If $X$ is infinite dimensional, then there exists a closed \emph{bounded} $C \subset X_+$, for which $\so(C)$ is not closed. Indeed, find disjoint norm one elements $e_1, e_2, \ldots \in X_+$.
For $n \in \N$ let $y_n = \sum_{k=1}^n 2^{-k} e_k$ and $x_n = y_n + e_n$. Then clearly $\|x_n\| \leq 2$ for any $n$; further, $\|x_n - x_m\| \geq 1$ for any $n \neq m$, hence $C = \{x_1, x_2, \ldots\}$ is closed.
However, $y_n \in \so(C)$ for any $n$, and the sequence $(y_n)$ converges to $\sum_{k=1}^\infty 2^{-k} e_k \notin \so(C)$.
\end{remark}

However, under certain conditions we can show that the solid hull of a closed set is closed.

\begin{proposition}\label{p:conv_closed_KB}
A Banach lattice $X$ is reflexive if and only if, for any norm closed, bounded convex $C \subset X_+$, $\so(C)$ is norm closed. 
\end{proposition}

\begin{proof}
Support first $X$ is reflexive, and $C$ is a norm closed bounded convex subset of $X_+$.
Suppose $(x_n)$ is a sequence in $\so(C)$, which converges to some $x$ in norm; show that $x$ belongs to $\so(C)$ as well.
Clearly $|x_n| \rightarrow |x|$ in norm.  For each $n$ find $y_n \in C$ so that $|x_n| \leq y_n$.
By passing to a subsequence if necessary, we assume that the sequence $(y_n)$ converges to some $y \in X$ in the weak topology.
For convex sets, norm and weak closures coincide, hence $y$ belongs to $C$.
For each $n$, $\pm x_n \leq y_n$; passing to the weak limit gives $\pm x \leq y$, hence $|x| \leq y$.

Now suppose $X$ is not reflexive. By \cite[Theorem 4.71]{AB}, 
there exists a sequence of disjoint elements $e_i \in \sphere(X)_+$, equivalent to the natural basis of either $c_0$ or $\ell_1$.

First consider the $c_0$ case. Let $C$ be the closed convex hull of
$$
x_1 = \frac{e_1}{2}, \, \,
x_n = \big(1 - 2^{-n} \big) e_1 + \sum_{j=2}^n e_j \, \, (n \geq 2) .
$$
We shall show that any element of $C$ can be written as $c e_1 + \sum_{i=2}^\infty c_i e_i$, with $c < 1$ . This will imply that $\so(C)$ is not closed: clearly $e_1 \in \overline{\so(C)}$.

The elements of $\ch(x_1, x_2, \ldots)$ are of the form $\sum_{i=1}^\infty t_i x_i = c e_1 + \sum_{i=2}^\infty c_i e_i$; here, $t_i \geq 0$, $t_i \neq 0$ for finitely many values of $i$ only, and $\sum_i t_i = 1$. Note that $c_i = \sum_{j=i}^\infty t_i$ for $i \geq 2$ (so $c_i = 0$ eventually); for convenience, let $c_1 = \sum_{j=1}^\infty t_i = 1$. Then $t_i = c_i - c_{i+1}$; Abel's summation technique gives
$$
c = \sum_{i=1}^\infty \big( 1 - 2^{-i} \big) t_i = 1 - \sum_{i=1}^\infty 2^{-i} \big( c_i - c_{i+1} \big) = \frac12 + \sum_{j=2}^\infty 2^{-j} c_j .
$$

Now consider $x \in C$. Then $x$ is the norm limit of the sequence
$$  x^{(m)} = c^{(m)} e_1 + \sum_{i=2}^\infty c_i^{(m)} e_i \in \ch(x_1, x_2, \ldots)  ; $$
for each $m$, the sequence $(c_i^{(m)})$ has only finitely many non-zero terms, $c^{(m)} = \frac12 + \sum_{j=2}^\infty 2^{-j} c_j^{(m)}$, and for all $m,n \in \N$,  $|c_i^{(m)} - c_i^{(n)}| \leq \|x^{(m)} - x^{(n)}\|$. Thus, $x = c e_1 + \sum_{i=2}^\infty c_i e_i$, with
$c = \frac12 + \sum_{j=2}^\infty 2^{-j} c_j$.
As $0 \leq c_j \leq 1$, and $\lim_j c_j = 0$, we conclude that $c < 1$, as claimed.

Now suppose $(e_i)$ are equivalent to the natural basis of $\ell_1$.
Let $C$ be the closed convex hull of the vectors
$$
x_n = \big(1 - 2^{-n} \big) e_1 + e_n \, \, (n \geq 2) ,
$$
and show that $e_1 \in \overline{\so(C)} \backslash \so(C)$.
Note that
$$
C = \Big\{ \Big( \sum_{i=2}^\infty \big(1 - 2^{-n} \big) t_i \Big) e_1 + \sum_{i=2}^\infty t_i e_i : t_2, t_3, \ldots \geq 0 , \sum_{i=2}^\infty t_i = 1 \Big\} .
$$
Clearly $e_1$ belongs to $\overline{\so(C)}$, but not to $\so(C)$.
\end{proof}



\section{Separation by positive functionals}\label{s:separation}


Throughout the section, $X$ is a Banach lattice, equipped with a locally convex Hausdorff topology $\tau$.
This topology is called \emph{sufficiently rich} if the following conditions are satisfied:

\begin{enumerate}[(i)]
\item
The space $X^\tau$ of $\tau$-continuous functionals on $X$ is a Banach lattice (with lattice operations defined by Riesz-Kantorovich formulas).
\item
$X_+$ is $\tau$-closed.
\end{enumerate}

Note that (i) and (ii) together imply that positive $\tau$-continuous functionals separate points. That is, for every $x \in X \backslash \{0\}$ there exists $f \in X^\tau_+$ so that $f(x) \neq 0$.
Indeed, without loss of generality, $x_+ \neq 0$. Then $- x_+ \notin X_+$, hence there exists $f \in X^\tau_+$ so that $f(x_+) > 0$. By \cite[Proposition 1.4.13]{M-N}, there exists $g \in X^\tau_+$ so that $g(x_+) > f(x_+)/2$ and $g(x_-) < f(x_+)/2$. Then $g(x) > 0$.

Clearly, the norm and weak topologies are sufficiently rich; in this case, $X^\tau = X^*$.
The weak$^*$ topology on $X$, induced by the predual Banach lattice $X_*$, is sufficiently rich as well; then $X^\tau = X_*$.

\begin{proposition}[Separation]\label{p:separation2}
Suppose $\tau$ is a sufficiently rich topology on a Banach lattice $X$, and $A \subset X_+$ is a $\tau$-closed positive-solid bounded subset of $X_+$. Suppose, furthermore, $x \in X_+$ does not belong to $A$. Then there exists $f \in X^\tau_+$ so that $f(x) > \sup_{a \in A} f(a)$.
\end{proposition}

\begin{lemma}\label{l:max}
Suppose $A$ and $X$ are as above, and $f \in X^\tau$. Then $\sup_{a \in A} f(a)$ $= \sup_{a \in A} f_+(a)$.
\end{lemma}

\begin{proof}
Clearly $\sup_{a \in A} f(a) \leq \sup_{a \in A} f_+(a)$. To prove the reverse inequality, write $f = f_+ - f_-$, with $f_+ \wedge f_- = 0$. Fix $a \in A$; then
$$
0 = \big[ f_+ \wedge f_- \big](a) = \inf_{0 \leq x \leq a} \big( f_+(a-x) + f_-(x) \big) .
$$
For any $\vr > 0$ we can find $x \in A$ so that $f_+(a-x), f_-(x) < \vr$. Then $f_+(x) = f_+(a) - f_+(a-x) > f_+(a) - \vr$, and therefore, $f(x) = f_+(x) - f_-(x) > f_+(a) - 2\vr$. Now recall that $\vr > 0$ and $a \in A$ are arbitrary.
\end{proof}

\begin{proof}[Proof of Proposition \ref{p:separation2}]
Use Hahn-Banach Theorem to find $f$ strictly separating $x$ from $A$. By Lemma \ref{l:max}, $f_+$ achieves the separation as well.
\end{proof}

\begin{remark}\label{r:lattice_needed}
In this paper, we do not consider separation results on general ordered spaces.
Our reasoning will fail without lattice structure. For instance, Lemma \ref{l:max} is false when $X$ is not a lattice, but merely an ordered space.
Indeed, consider $X = M_2$ (the space of real $2 \times 2$ matrices), $\displaystyle f = \begin{pmatrix} 1 & 0 \\ 0 & -1 \end{pmatrix}$, and $A = \{t a_0 : 0 \leq t \leq 1\}$, where
$\displaystyle a_0 = \begin{pmatrix} 1 & 1 \\ 1 & 1 \end{pmatrix}$; one can check that $A = \{x \in M_2 : 0 \leq x \leq a_0\}$. Then $f|_A = 0$, while $\displaystyle \sup_{x \in A} f_+(x) = 1$. 

The reader interested in the separation results in the non-lattice ordered setting is referred to 
an interesting result of \cite{FW}, recently re-proved in \cite{Ama}.
\end{remark}

\section{Solid convex hulls: theorems of Krein-Milman and Milman}\label{s:KM}

Throughout this section, the topology $\tau$ is assumed to be sufficiently rich (defined in the beginning of Section \ref{s:separation}).

\begin{theorem}[``Solid'' Krein-Milman]\label{t:KM}
Any $\tau$-compact positive-solid subset $A$ of $X_+$ coincides with the $\tau$-closed positive-solid convex hull of its order extreme points.
\end{theorem}

\begin{proof}
Let $A$ be a $\tau$-compact positive-solid subset of $X_+$.
Denote the $\tau$-closed positive convex hull of $\oep(A)$ by $B$; then clearly $B \subset A$.
The proof of the reverse inclusion is  similar to that of the ``usual'' Krein-Milman.

Suppose $C$ is a $\tau$-compact subset of $X$.
We say that a non-void closed $F \subset C$ is an \emph{order extreme subset} of $C$ if, whenever 
$x \in F$ and $a_1, a_2 \in C$ satisfy $x \leq (a_1 + a_2)/2$, then necessarily $a_1, a_2 \in F$.
The set ${\mathcal{F}}(C)$ of order extreme subsets of $C$ can be ordered by reverse inclusion
(this makes $C$ the minimal order extreme subset of itself). By compactness, each chain has an upper bound; therefore, by Zorn's Lemma, ${\mathcal{F}}(C)$ has a maximal element.
We claim that these maximal elements are singletons, and they are the order extreme points of $C$.

We need to show that, if $F \in {\mathcal{F}}(C)$ is not a singleton, then there exists $G \subsetneq F$ which is also an order extreme set. To this end, find distinct $a_1, a_2 \in F$, and $f \in X^\tau_+$ which separates them -- say $f(a_1) > f(a_2)$. Let $\alpha = \max_{x \in F} f(x)$, then $G = F \cap f^{-1}(\alpha)$ is a proper, order extreme subset of $F$. 

Suppose, for the sake of contradiction, that there exists $x \in A \backslash B$. Use Proposition \ref{p:separation2} to find $f \in X^\tau_+$ so that $f(x) > \max_{y \in B} f(y)$. Let $\alpha = \max_{x \in A} f(x)$, then $A \cap f^{-1}(\alpha)$ is an order extreme subset of $A$, disjoint from $B$. As noted above, this subset contains at least one extreme point. This yields a contradiction, as we started out assuming all order extreme points lie in $B$.
\end{proof}

\begin{corollary}\label{c:KM}
 Any $\tau$-compact solid subset of $X$ coincides with the $\tau$-closed solid convex hull of its order extreme points.
\end{corollary}

Of course, there exist Banach lattices whose unit ball has no order extreme points at all -- $L_1(0,1)$, for instance. However, an order analogue of \cite[Lemma 1]{Linl1} holds.

\begin{proposition}\label{p:some_vs_all}
For a Banach lattice $X$, the following two statements are equivalent:
\begin{enumerate}
 \item 
 Every bounded closed solid convex subset of $X$ has an order extreme point.
 \item
 Every bounded closed solid convex subset of $X$ is the closed solid convex hull of its order extreme points.
\end{enumerate}
\end{proposition}

\begin{proof}
(2) $\Rightarrow$ (1) is evident; we shall prove (1) $\Rightarrow$ (2).
Suppose $A \subset X$ is closed, bounded, convex, and solid. Let
$B = \csch(\oep(A))$
(which is not empty, by (1)). Suppose, for the sake of contradiction, that $B$ is a proper subset of $A$. Let $a \in A \backslash B$.  Since $B$ and $A$ are solid, $|a| \in A \backslash B$ as well, so without loss of generality we assume that $a \geq 0$. Then there exists $f \in \sphere(X^*)_+$ which strictly separates $a$ from $B$; consequently,
$$
\sup_{x \in A} f(x) \geq f(a) > \sup_{x \in B} f(x) .
$$
Fix $\varepsilon > 0$ so that
$$
2 \sqrt{2 \varepsilon} \alpha < \sup_{x \in A} f(x) - \sup_{x \in B} f(x) , \, \,
{\textrm{where}} \, \, \alpha = \sup_{x \in A} \|x\| .
$$
By Bishop-Phelps-Bollob\'as Theorem (see e.g. \cite{Boll} or \cite{CKMetc}), there exists $f' \in \sphere(X^*)$, attaining its maximum on $A$, so that $\|f - f'\| \leq \sqrt{2 \varepsilon}$.

Let $g = |f'|$, then $\|f - g\| \leq \|f - f'\| \leq \sqrt{2 \varepsilon}$. Further, $g$ attains its maximum on $A_+$, and $\max_{g \in A} g(x) > \sup_{x \in B} g(x)$.
Indeed, the first statement follows immediately from the definition of $g$. To establish the second one, note that the triangle inequality gives us
$$
\sup_{x \in B} g(x) \leq \sqrt{2 \varepsilon} \alpha + \sup_{x \in B} f(x) , \, \, \,
\sup_{x \in A} g(x) \geq \sup_{x \in A} f(x) - \sqrt{2 \varepsilon} \alpha . 
$$
Our assumption on $\varepsilon$ gives us $\max_{g \in A} g(x) > \sup_{x \in B} g(x)$.

Let $D = \{a \in A : g(a) = \sup_{x \in A} g(x)\}$. Due to (1), $D$ has an order extreme point which is also an order extreme point of $A$; this point lies inside of $B$, leading to the desired contradiction.
\end{proof}


Milman's theorem \cite[3.25]{Rud} states that, if both $K$ and $\overline{\ch(K)}^\tau$ are compact,
then $\ep\big(\overline{\ch(K)}^\tau\big) \subset K$.  An order analogue of Milman's theorem exists:


\begin{theorem}\label{t:order-milman}
Suppose $X$ is a Banach lattice.
\begin{enumerate}
\item
If $K \subset X_+$ and $\overline{\ch(K)}^\tau$ are $\tau$-compact, then $\oep\big(\overline{\sch(K)}^\tau\big) \subseteq K$.
\item
If $K \subset X_+$ is weakly compact, then $\oep(\csch(K)) \subseteq K$.
\item
If $K \subset X$ is norm compact, then $\oep(\csch(K)) \subseteq |K|$.
\end{enumerate}
\end{theorem}


The following lemma describes the solid hull of a $\tau$-compact set.

\begin{lemma}\label{l:compact_closed}
Suppose a Banach lattice $X$ is equipped with a sufficiently rich topology $\tau$.
If $C \subset X_+$ is $\tau$-compact, then $\so(C)$ is $\tau$-closed.
\end{lemma}

\begin{proof}
Suppose a net $(y_i) \subset \so(C)$ $\tau$-converges to $y \in X$. For each $i$ find $x_i \in C$ so that $|y_i| \leq x_i$ -- or equivalently, $y_i \leq x_i$ and $-y_i \leq x_i$. Passing to a subnet if necessary, we assume that $x_i \to x \in C$ in the topology $\tau$. Then $\pm y \leq x$, which is equivalent to $|y| \leq x$.
\end{proof}

\begin{proof}[Proof of Theorem \ref{t:order-milman}]
(1) We first consider a $\tau$-compact $K\subseteq X_+$.
Milman's traditional theorem holds that $\ep\big(\overline{\ch(K)}^\tau\big) \subseteq K$.  Every order extreme point of a set is extreme, hence the order extreme points of $\overline{\ch(K)}^\tau$ are in $K$. 
Therefore, by Lemma \ref{l:compact_closed} and Corollary \ref{c:equal-sch},
\[\overline{\sch(K)}^\tau = \overline{\so(\ch(K))}^\tau \subseteq \so\big(\overline{\ch(K)}^\tau\big) = \{x: |x| \leq y \in \overline{\ch(K)}^\tau \} . \] 
Thus, the points of $\overline{\sch(K)}^\tau \backslash \overline{\ch(K)}^\tau$  cannot be order extreme due to being dominated by $\overline{\ch(K)}^\tau$.
Therefore $\oep\big(\overline{\sch(K)}^\tau\big) \subseteq \oep\big(\overline{\ch(K)}^\tau\big) \subseteq K$.

(2) Combine (1) with Krein's Theorem (see e.g.~\cite[Theorem 3.133]{FHHMZ}), which states that $\overline{\ch(K)}^w = \overline{\ch(K)}$ is weakly compact.

(3) Finally, suppose $K\subseteq X$ is norm compact.  By Corollary \ref{c:equal-sch}, $\csch(K) = \csch(|K|)$. $|K|$ is norm compact, hence by \cite[Theorem 3.20]{Rud}, so is $\overline{\ch(|K|)}$. By the proof of part (1),  $\oep ( \csch(K) ) \subseteq |K|$.
\end{proof}

We turn our attention to
interchanging ``solidification'' and norm closure.
We work with the norm topology, unless specified otherwise.

\begin{lemma}\label{l:switch-close-solid}
	Let $C\subseteq X$, where $X$ is a Banach lattice, and suppose that $\so(\overline{|C|})$ is closed.  Then $ \overline{\so(C)} = \so(\overline{|C|})$. 
\end{lemma}

\begin{proof}	
One direction is easy: $\so(C)= \so(|C|) \subseteq \so(\overline{|C|})$, hence $\overline{\so(C)} \subseteq \overline{\so(\overline{|C|})} = \so(\overline{|C|})$. 

Now consider $x\in \so(\overline{|C|})$.  Then by definition, $|x| \leq y$ for some $y \in \overline{|C|}$.  Take $y_n \in |C|$ such that $y_n \rightarrow y$ .  Then $|x|\wedge y_n \in \so(|C|) = \so(C)$ for all $n$. Furthermore, \[ |x_+\wedge y_n -x_-\wedge y_n| =|x|\wedge y_n, \]
so, $x_+ \wedge y_n - x_-\wedge y_n \in \so(C)$.  By norm continuity of $\wedge$,  \[x_+\wedge y_n -x_-\wedge y_n \rightarrow x_+\wedge y -x_-\wedge y = x, \] hence $x \in \overline{\so(C)}$.
\end{proof}

\begin{remark}
	The assumption  of $\so(\overline{|C|})$ being closed is necessary: Remark \ref{r:shadow-not-closed} shows that, for a closed $C \subset X_+$, $\so(C)$ need not be closed.
\end{remark}

\begin{corollary}\label{c:switch-close-solid1}
	Suppose $C\subseteq X$is relatively compact in the norm topology.
	Then $\overline{\so(C)} = \so(\overline{C})$.
\end{corollary}

\begin{proof}
The set $\overline{C}$ is compact, hence, by the continuity of $| \cdot |$, the same is true for $|\overline{C}|$. Consequently, $|\overline{C}| \subseteq \overline{|C|} \subseteq \overline{|\overline{C}|} = |\overline{C}|$, hence $|\overline{C}| = \overline{|C|}$.
By Lemmas \ref{l:compact_closed} and \ref{l:switch-close-solid}, $\so(\overline{C}) = \so(|\overline{C}|) = \so(\overline{|C|}) = \overline{\so(C)}$.
\end{proof}

\begin{remark}\label{r:closure_versus_modulus}
In the weak topology, the equality $|\overline{C}| = \overline{|C|}$ may fail. Indeed, equip the Cantor set $\Delta = \{0,1\}^\N$ with its uniform probability measure $\mu$.
Define $x_i \in L_2(\mu)$ by setting, for $t = (t_1, t_2, \ldots) \in \Delta$, $x_i(t) = t_i - 1/4$ (that is, $x_i$ equals to either $3/4$ or $-1/4$, depending on whether $t_i$ is $1$ or $0$).
Then $C = \{x_i : i \in \N\}$ belongs to the unit ball of $L_2(\mu)$, hence it is relatively compact. It is clear that $\overline{C}$ contains $\one/4$ (here and below, $\one$ denotes the constant $1$ function).
On the other hand, $\overline{C}$ does not contain $\one/2$, which can be witnessed by applying the integration functional. Conversely, $\overline{|C|}$ contains $\one/2$, but not $\one/4$.
\end{remark}

\begin{remark}\label{r:weakly_compact_solids}
Relative weak compactness of solid hulls have been studied before. If $X$ is a Banach lattice, then,
by \cite[Theorem 4.39]{AB}, 
it is order continuous iff the solid hull of any weakly compact subset of $X_+$ is relatively weakly compact. Further, by \cite{ChW}, the following three statements are equivalent:
\begin{enumerate}
 \item 
 The solid hull of any relatively weakly compact set is relatively weakly compact.
 \item
 If $C \subset X$ is relatively weakly compact, then so is $|C|$.
 \item
 $X$ is a direct sum of a KB-space and a purely atomic order continuous Banach lattice (a Banach lattice is called purely atomic if its atoms generate it, as a band).
\end{enumerate}
\end{remark}


Finally, we return to the connections between extreme points and order extreme points. As noted in the paragraph preceding Theorem \ref{t:connection}, a non-zero extreme point of a positive-solid set need not be order extreme. However, we have:

\begin{proposition}\label{p:under_oep}
Suppose $\tau$ is a sufficiently rich topology,
and $A$ is a $\tau$-compact positive-solid convex subset of $X_+$. Then for any extreme point $a \in A$ there exists an order extreme point $b \in A$ so that $a \leq b$.
\end{proposition}

\begin{remark}\label{r:for_domination_need_compactness}
The compactness assumption is essential. Consider, for instance, the closed set $A \subset C[-1,1]$, consisting of all functions $f$ so that $0 \leq f \leq \one$, and $f(x) \leq x$ for $x \geq 0$. Then $g(x) = x \vee 0$ is an extreme point of $A$; however, $A$ has no order extreme points.
\end{remark}

\begin{proof}
If $a$ is not an order extreme point, then we can find distinct $x_1, x_2 \in A$ so that $2a \leq x_1 + x_2$.
Then $2a \leq (x_1 + x_2) \wedge (2a) \leq x_1 \wedge (2a) + x_2 \wedge (2a) \leq x_1 + x_2$.
Write $2a = x_1 \wedge (2a) + (2a - x_1 \wedge (2a))$. Both summands are positive, and both belong to $A$ (for the second summand, note that $2a - x_1 \wedge (2a) \leq x_2$). Therefore, $x_1 \wedge (2a) = a = 2a - x_1 \wedge (2a)$, 
hence in particular $x_1 \wedge (2a) = a$.
Similarly, $x_2 \wedge (2a) = a$.
Therefore, we can write $x_1$ as a disjoint sum $x_1 = x_1' + a$ ($a, x_1'$ are quasi-units of $x_1$). In the same way, $x_2 = x_2' + a$ (disjoint sum).


Now consider the $\tau$-closed set $B = \{x \in A : x \geq a\}$. As in the proof of Theorem \ref{t:KM}, we show that the family of $\tau$-closed extreme subsets of $B$ has a maximal element; moreover, such an element is a singleton $\{b\}$.
It remains to prove that $b$ is an order extreme point of $A$. Indeed, suppose $x_1, x_2 \in A$ satisfy $2b \leq x_1 + x_2$. A fortiori, $2a \leq x_1 + x_2$, hence, by the preceding paragraph, $x_1, x_2 \in B$. Thus, $x_1 = b = x_2$.
%
\end{proof}

\section{Examples: AM-spaces and their relatives}\label{s:examples}


The following example shows that, in some cases, $\ball(X)$ is much larger than the closed convex hull of its extreme points, yet is equal to the closed solid convex hull of its order extreme points.

\begin{proposition}\label{p:fin_many}
For a Banach lattice $X$, $\ball(X)$ is the (closed) solid convex hull of $n$ disjoint non-zero elements if and only if $X$ is lattice isometric to $C(K_1) \oplus_1 \ldots \oplus_1 C(K_n)$ for suitable non-trivial Hausdorff compact topological spaces $K_1,...,K_n$.
\end{proposition}

\begin{proof}
Clearly, the only order extreme points of $\ball(C(K_1) \oplus_1 \ldots \oplus_1 C(K_n))$ are $\one_{K_i}$, with $1 \leq i \leq n$.

Conversely, suppose $\ball(X) = \csch(x_1, \ldots, x_n)$, where $x_1, \ldots, x_n \in \ball(X)_+$ are disjoint.
It is easy to see that, in this case, $\ball(X) = \sch(x_1, \ldots, x_n)$.
Moreover, $x_i \in \sphere(X)_+$ for each $i$. Indeed, otherwise there exists $i \in \{1, \ldots, n\}$ and $\lambda > 1$ so that $\lambda x_i \in \sch(x_1, \ldots, x_n)$, or in other words, $\lambda x_i \leq \sum_{j=1}^n t_j x_j$, with $t_j \geq 0$ and $\sum_j t_j \leq 1$. Consequently, due to the disjointness of $x_j$'s,
$$
\lambda x_i = (\lambda x_i) \wedge (\lambda x_i) \leq \big( \sum_{j=1}^n t_j x_j \big) \wedge (\lambda x_i) \leq
\sum_{j=1}^n (t_j x_j) \wedge (\lambda x_i) \leq t_i x_i ,
$$
which yields the desired contradiction.

Let $E_i$ be the ideal of $X$ generated by $x_i$, meaning the set of all $x \in X$ for which there exists $c > 0$ so that $|x| \leq c |x_i|$. Note that, for such $x$, $\|x\|$ is the infimum of all $c$'s with the above property.
Indeed, if $|x| \leq |x_i|$, then clearly $x \in \ball(X)$. Conversely, suppose $x \in \ball(X) \cap E_i$. In other words, $|x| \leq c x_i$ for some $c$, and also $|x| \leq \sum_j t_j x_j$, with $t_j \geq 0$, and $\sum_j t_j = 1$. Then $|x| \leq (c x_i) \wedge ( \sum_j t_j x_j ) = (c \wedge t_i) x_i$.
Consequently, $E_i$ (with the norm inherited from $X$) is an $AM$-space, whose strong unit is $x_i$. By \cite[Theorem 2.1.3]{M-N}, $E_i$ can be identified with $C(K_i)$, for some Hausdorff compact $K_i$. 

Further, Proposition \ref{p:interchange} shows that $X$ is the direct sum of the ideals $E_i$: any $y \in X$ has a unique disjoint decomposition $y = \sum_{i=1}^n y_i$, with $y_i \in E_i$. We have to show that $\|y\| = \sum_i \|y_i\|$.
Indeed, suppose $\|y\| \leq 1$.  Then $|y| = \sum_i |y_i| \leq \sum_j t_j x_j$, with $t_j \geq 0$, and $\sum_j t_j = 1$. Note that $\|y_i\| \leq 1$ for every $i$, or equivalently, $|y_i| \leq x_i$. Therefore,
$$
|y_i| = |y| \wedge x_i = \big( \sum_j t_j x_j \big) \wedge x_i = t_i ,
$$
which leads to $\|y_i\| \leq t_i$; consequently, $\|y\| \leq \sum_i t_i \leq 1$.
\end{proof}

\begin{example}\label{e:other_etreme_points}
For $X = (C(K_1) \oplus_1 C(K_2)) \oplus_\infty C(K_3)$, order extreme points of $\ball(X)$ are $\one_{K_1} \oplus_\infty \one_{K_3}$ and $\one_{K_2} \oplus_\infty \one_{K_3}$; $\ball(X)$ is the solid convex hull of these points. Thus, the word ``disjoint'' in the statement of Proposition \ref{p:fin_many} cannot be omitted.
\end{example}


Note that $\ball(C(K))$ is the closed solid convex hull of its only order extreme point -- namely, $\one_K$. This is the only type of AM-spaces with this property.

\begin{proposition}\label{p:describe_AM}
Suppose $X$ is an AM-space, and $\ball(X)$ is the closed solid convex hull of finitely many of its elements. Then $X = C(K)$ for some Hausdorff compact $K$.
\end{proposition}

\begin{proof}
Suppose $\ball(X)$ is the closed solid convex hull of $x_1, \ldots, x_n \in \ball(X)_+$. Then $x_0 := x_1 \vee \ldots \vee x_n \in \ball(X)_+$ (due to $X$ being an AM-space), hence $x \in \ball(X)$ iff $|x| \leq x_0$. Thus, $x_0$ is the strong unit of $X$.
\end{proof}

\begin{proposition}\label{p:only_C(K)_has_OEP}
If $X$ is an AM-space, and $\ball(X)$ has an order extreme point, then $X$ is lattice isometric to $C(K)$, for some Hausdorff compact $K$.
\end{proposition}

\begin{proof}
Suppose $a$ is order extreme point of $\ball(X)$. We claim that $a$ is a strong unit, which means that $a \geq x$ for any $x \in \ball(X)_+$. Suppose, for the sake of contradiction, that the inequality $a \geq x$ fails for some $x \in \ball(X)_+$. Then $b = a \vee x \in \ball(X)_+$ (due to the definition of an AM-space), and $a \leq (a+b)/2$, contradicting the definition of an order extreme point.
\end{proof}

We next consider norm-attaining functionals.
%
%
%
It is known that, for a Banach space $X$, any element of $X^*$ attains its norm iff $X$ is reflexive. 
If we restrict ourself to positive functionals on a Banach lattice, the situation is different: clearly every positive functional on $C(K)$ attains it norm at $\one$. Below we show that, among separable AM-spaces, only $C(K)$ has this property.

\begin{proposition}\label{p:only_C(K)_attain_norm}
Suppose $X$ is a separable AM-space, so that every positive linear functional attains its norm. Then $X$ is lattice isometric to $C(K)$.
\end{proposition}

\begin{proof}
Let $(x_i)_{i=1}^\infty$ be a dense sequence in $\sphere(X)_+$. For each $i$ find $x_i^* \in \ball(X^*_+)$ so that $x_i^*(x_i) = 1$. Let $x^* = \sum_{i=1}^\infty 2^{-i} x_i^*$. We shall show that $\|x^*\| = 1$. Indeed, $\|x^*\| \leq \sum_i 2^{-i} = 1$ by the triangle inequality. For the opposite inequality, fix $N \in \N$, and let $x = x_1 \vee \ldots \vee x_N$. Then $x \in \sphere(X)_+$, and 
$$
\|x^*\| \geq x^*(x) \geq \sum_{i=1}^N 2^{-i} x_i^*(x) \geq \sum_{i=1}^N 2^{-i} x_i^*(x_i) = \sum_{i=1}^N 2^{-i}   = 1 - 2^{-N} .
$$
As $N$ can be arbitrarily large, we obtain the desired estimate on $\|x^*\|$.

Now suppose $x^*$ attains its norm on $a \in \sphere(X)_+$. We claim that $a$ is the strong unit for $X$. Suppose otherwise; then there exists $y \in \ball(X)_+$ so that $a \geq y$ fails. Let $b = a \vee y$, then $z = b - y$ belongs to $X_+\backslash\{0\}$. Then $1 \geq x^*(b) \geq x^*(a) = 1$, hence $x^*(z) = 0$. However, $x^*$ cannot vanish at $z$. Indeed, find $i$ so that $\|z/\|z\| - x_i\| < 1/2$. Then $x_i^*(z) \geq \|z\|/2$, hence $x(z) > 2^{-i-1} \|z\| > 0$. This gives the desired contradiction.
\end{proof}

In connection to this, we also mention a result about norm-attaining functionals on order continuous Banach lattices.

\begin{proposition}\label{p:functional_OC}
An order continuous Banach lattice $X$ is reflexive if and only if every positive linear functional on it attains its norm.
\end{proposition}

\begin{proof}
If an order continuous Banach lattice $X$ is reflexive, then clearly every linear functional is norm-attaining. If $X$ is not reflexive, then,
by the classical result of James, there exists $x^* \in X^*$ which does not attain its norm. We show that $|x^*|$ does not either.

Let $B_+ = \{x \in X : x^*_+(|x|) = 0\}$, and define $B_-$ similarly. As all linear functionals on $X$ are order continuous \cite[Section 2.4]{M-N}, $B_+$ and $B_-$ are bands \cite[Section 1.4]{M-N}. Due to the order continuity of $X$ \cite[Section 2.4]{M-N}, $B_{\pm}$ are  ranges of band projections $P_{\pm}$. Let $B$ be the range of $P = P_+ P_-$; let $B_+^o$ be the range of $P_+^o = P_+ P_-^\perp = P_+ - P$ (where we set $Q^\perp = I_X - Q$), and similarly for $B_-^o$ and $P_-^o$. Note that $P_+^o + P_-^o = P^\perp$.

Suppose for the sake of contradiction that $x \in \sphere(X)_+$ satisfies $|x^*|(x) = \|x^*\|$. Replacing $x$ by $P^\perp x$ if necessary, we assume that $P x = 0$, so $x = P_+^o x + P_-^o x$. Then $\|P_+^o x - P_-^o x\| = 1$, and
\begin{align*}
x^* \big( P_-^o x - P_+^o x \big)
&
= x^*_+ \big( P_-^o x \big) - x^*_+ \big( P_+^o x \big) - x^*_- \big( P_-^o x \big) + x^*_- \big( P_+^o x \big) \\
&
= x^*_+ \big( P_-^o x \big) + x^*_- \big( P_+^o x \big) = |x^*|(x) = \|x^*\| ,
\end{align*}
which contradicts our assumption that $x$ does not attain its norm.
\end{proof}

\section{On the number of order extreme points}\label{s:how_many_extreme_points}

It is shown in \cite{LP} that, if a Banach space $X$ is reflexive and infinite-dimensional Banach lattice, then $\ball(X)$ has uncountably many extreme points. Here, we establish a similar lattice result.

\begin{theorem}\label{t:uncount_many}
If $X$ is a reflexive infinite-dimensional Banach lattice, then $\ball(X)$ has uncountably many order extreme points.
\end{theorem}

Note that if $X$ is a reflexive infinite-dimensional Banach lattice, then Theorems \ref{t:connection} and \ref{t:uncount_many} imply that $\ball(X)$ has uncountably many extreme points, re-proving the result of \cite{LP} in this case.

\begin{proof}
Suppose, for the sake of contradiction, that there were only countably many such points $\{x_n\}$.  For each such $x_n$, we define 
$F_n =\{ f \in  \ball(X^*)_+ : f(x_n) = \| f\| \}$.
Clearly $F_n$ is weak$^*$ ($=$ weakly) compact.

By the reflexivity of $X$, any $f\in  \ball(X^*)$ attains its norm at some $x \in \ep(\ball(X))$. Since $f(x) \leq |f|(|x|)$ we assume that any positive functional attains its norm at a positive extreme point in $ \ball(X)$.  By Theorem \ref{t:connection}, these are precisely the order extreme points.  Therefore $\bigcup F_n =  \ball(X^*)_+$.  By the Baire Category Theorem, one of these sets $F_n$ must have non-empty interior in $ \ball(X^*)_+$.

Assume it is $F_1$.  Pick $f_0\in F_1 $, and $y_1,...,y_k \in X$, such that if $f \in  \ball(X^*)_+$ and for each $y_i$, $|f(y_i) - f_0(y_i) | < 1$, then $f \in F_1$.  Without loss of generality, we assume that $\|f_0 \| < 1$, and also that each $y_i \geq 0$.

Further, we can and do assume that there exist mutually disjoint $u_1, u_2, \ldots \in \sphere(X)_+$ which are disjoint from $y = \vee_i y_i$.
Indeed, find mutually disjoint $z_1, z_2, \ldots \in \sphere(X)_+$.  Denote the corresponding band projections by $P_1, P_2, \ldots$ (such projections exist, due to the $\sigma$-Dedekind completeness of $X$).
Then the vectors $P_n y$ are mutually disjoint, and dominated by $y$. As $X$ is reflexive, it must be order continuous, and therefore, $\lim_n \|P_n y\| = 0$. Find $n_1 < n_2 < \ldots$ so that $\sum_j \|P_{n_j} y\| < 1/2$. Let $w_i = \sum_j P_{n_j} y_i$ and $y'_i = 2(y_i - w_i)$.
Then if $|(f_0-g)(y'_i)|< 1$, with $g\geq 0, \|g\| \leq 1$, it follows that 
\begin{align*}
|(f_0 - g)(y_i)| & \leq \frac{1}{2} ( |(f_0 - g)(y'_i) | + |(f_0- g)(w_i) |) \\
&\leq \frac{1}{2} ( 1 + \|f_0 - g\| \|w_i\| ) <  \frac12(1+ 2\cdot \frac12 ) = 1
\end{align*}
We can therefore replace $y_i$ with $y'_i$ to ensure sufficient conditions for being in $F_1$.
Then the vectors $u_j = z_{n_j}$ have the desired properties.
Let $P$ be the band projection complementary to $\sum_j P_{n_j}$ (in other words, complementary to the the band projection of $\sum_j 2^{-j} u_j$); then $P y_i = y_i$ for any $i$.

By \cite[Lemma 1.4.3 and its proof]{M-N}, there exist linear functionals $g_j \in \sphere(X^*)_+$ so that $g_j(u_j) = 1$, and $g_j =P_{n_j}^* g_j$. Consequently, the functionals $g_j$ are mutually disjoint, 
and $g_j|_{\ran P} = 0$. 
For $j \in \N$ find $\alpha_j \in [1 - \|P^* f_0\|, 1]$ so that $\|f_j\| = 1$, where $f_j = P^* f_0 + \alpha_j g_j$. Then, for $1 \leq i \leq k$, $f_j(y_i) = (P^* f_0)(y_i) + \alpha_j g_j(y_i) = f_0(y_i)$, which implies that, for every $j$, $f_j$ belongs to $F_1$, hence attains its norm at $x_1$. 

On the other hand, note that $\lim_j g_j(x_1) = 0$. Indeed, otherwise, there exist $\gamma > 0$ and a sequence $(j_k)$ so that $g_{j_k}(x_1) \geq \gamma$ for every $k$. For any finite sequence of positive numbers $(\beta_k)$, we have
$$
\sum_k |\beta_k| \geq \big\| \sum_k \beta_k g_{j_k} \big\|  \geq \sum_k \beta_k g_{j_k} (x_1) \geq \gamma \sum_k |\beta_k| .
$$
As the functionals $g_{j_k}$ are mutually disjoint, the inequalities
$$
\sum_k |\beta_k| \geq \big\| \sum_k \beta_k g_{j_k} \big\| \geq \gamma \sum_k |\beta_k| 
$$
hold for every finite sequence $(\beta_k)$. We conclude that $\overline{\spn}[g_{j_k} : k \in \N]$ is isomorphic to $\ell_1$, which contradicts the reflexivity of $X$. Thus, $\lim_j g_j(x_1) = 0$, hence $\lim_j f_j (x_1) = f_0(P x_1) \leq \|f_0\| < 1$.
\end{proof}

\begin{corollary}
	Suppose $C$ is a closed, bounded, solid, convex subset of a reflexive Banach lattice, having non-empty interior.  Then $C$ contains uncountably many order extreme points.
\end{corollary}

\begin{proof}
We assume without loss of generality that $\sup_{x \in C} \|x\| = 1$. Note that $0$ is an interior point of $C$. Indeed, suppose $x$ is an interior point. Pick $\varepsilon > 0$ such that $x + \vr \ball(X) \subset C$.
For any $k$ such that $\|k\| < \varepsilon$, we have $\frac{k}{2} = \frac{-x}{2} +\frac{x+k}{2} \in C$, since $C$ is solid and convex.  Hence $\frac{\varepsilon}{2} \ball(X) \subseteq C$.  Since $C$ is bounded, we can then define an equivalent norm, with $\|y\|_C = \inf \{\lambda > 0: y \in \lambda C  \}$.  Since $C$ is solid, $\|y \|_C = \| \  |y | \  \|_C$, and the norm is consistent with the order.  Finally, $\| \cdot \|_C$ is equivalent to $\| \cdot \|$, since for all $y\in X$, we have that $\frac{\vr}{2}\|y\|_C \leq \|y\| \leq \|y\|_C$.  The conclusion follows by Theorem \ref{t:uncount_many}.
\end{proof}


\section{The solid Krein-Milman Property and the RNP}\label{s:SKMP}

We say that a Banach lattice (or, more generally, an ordered Banach space) $X$ has the \emph{Solid Krein-Milman Property} (\emph{SKMP}) if every solid closed boun\-ded subset of $X$ is the closed solid convex hull of its order extreme points.
This is analogous to the canonical Krein-Milman Property (KMP) in Banach spaces, which is defined in the similar manner, but without any references to order. It follows from Theorem \ref{t:connection} that the KMP implies the SKMP.

These geometric properties turn out to be related to the Radon-Nikod{\'ym} Property (RNP).
It is known that the RNP implies the KMP, and, for Banach lattices, the converse is also true (see \cite{Cas} for a simple proof).
For more information about the RNP in Banach lattices, see \cite[Section 5.4]{M-N}; a good source of information about the RNP in general is \cite{Bour} or \cite{DU}.

One of the equivalent definitions of the RNP of a Banach space $X$ involves integral representations of operators $T : L_1 \to X$. If $X$ is a Banach lattice, then, by \cite[Theorem IV.1.5]{Sch}, any such operator is regular (can be expressed as a difference of two positive ones); so positivity comes naturally into the picture.

\begin{theorem}\label{t:RNP}
	For a Banach lattice $X$, the SKMP, KMP, and RNP are equivalent.
\end{theorem}

\begin{proof}
The implications RNP $\Leftrightarrow$ KMP $\Rightarrow$ SKMP are noted above.
Now suppose $X$ fails the RNP (equivalently, the KMP). We shall establish the failure of the SKMP in two different ways, depending on whether $X$ is a KB-space, or not.

(1) If $X$ is not a KB-space, then \cite[Theorem 2.4.12]{M-N} there exist disjoint $e_1, e_2, \ldots \in \sphere(X)_+$, equivalent to the canonical basis of $c_0$. Then the set
$$
C = \overline{\so \Big( \big\{ \sum_i \alpha_i e_i : \max_i |\alpha_i| = 1 , \, \lim_i \alpha_i = 0 \big\} \Big)}
$$
is solid, bounded, and closed. To give a more intuitive description of $C$, for $x \in X$ we let $x_i = |x| \wedge e_i$. It is easy to see that $x \in C$ if and only if $\lim_i \|x_i\| = 0$, and $|x| = \sum_i x_i$. 
Finally, show that $x \in C_+$ cannot be an order extreme point.
Find $i$ so that $\|x_i\| < 1/2$, and consider $x' = \sum_{j \neq i} x_j + e_i$. Then clearly $x' \in C$,  and $x' - x \in X_+ \backslash \{0\}$.


(2) If $X$ is a KB-space failing the RNP, then, by \cite[Proposition 5.4.9]{M-N}, $X$ contains a separable sublattice $Y$ failing the RNP.
Find a quasi-interior point $u \in Y$ -- that is, $y = \lim_n y \wedge (nu)$ for any $y \in Y_+$.
By \cite[Corollary 5.4.20]{M-N}, $Y$ is not order dentable -- that is, $Y_+$ contains a non-empty convex bounded subset $A$ so that, for every $n \in \N$, $A = \overline{\ch(A \backslash H_n)}$, where $H_n = \{ y \in Y_+ : \|u \wedge y\| > \frac1n \}$.

We use the techniques (and notation) of \cite{BourTal} to construct a set $C$ witnessing the failure of the SKMP. 
For $f\in Y^*$, let $M(A,f) = \sup_{x\in A} |f(x)|$. For $\alpha > 0$, define the \emph{slice} $T(A,f,\alpha) = \{ x\in A: f(x) > M(A,f) - \alpha\}$.
By  \cite{BourTal}, we can construct increasing measure spaces $\Sigma_n$ on $[0,1]$ with $|\Sigma_n|$ finite, as well as $\Sigma_n$-measurable functions $Y_n:[0,1] \rightarrow A$, $f_n:[0,1]\rightarrow Y^*$, and $\alpha_n:[0,1] \rightarrow \R$ such that:
\begin{enumerate}
	\item For any $n$ and $t$, $Y_n(t) \in \overline{T(A, f_n(t), \alpha_n(t))}$.
	\item $(Y_n)$ is a martingale -- that is, $Y_n(t) = {\mathbb{E}}^{\Sigma_n}(Y_{n+1}(t))$, for any $t$ and $n$ (${\mathbb{E}}$ stands for the conditional expectation).
	\item For any $n$ and $t$, $H_n \cap \overline{T(A,f_n(t),\alpha_n(t)))} = \emptyset$. 
	\item For any $n$ and $t$, $T(A,f_{n+1}(t), \alpha_{n+1}(t)) \subseteq T(A,f_n(t), \alpha_n(t))$.
\end{enumerate}

Now let $C' = \overline{\ch(\{Y_n(t), n\in \N, t\in [0,1] \})}$, then the set $C = \overline{\so(C')}$ (the solid hull is in $X$) is closed, bounded, convex, and solid.
We will show that $C$ has no order extreme points.   
By Theorem \ref{t:connection}, it suffices to show that no $x \in C_+ \backslash \{0\}$ can be an extreme point of $C$, or equivalently, of $C_+ = C \cap X_+$.

From now on, fix $x \in C_+ \backslash \{0\}$. 
Note that $x \wedge u \neq 0$. Indeed, suppose, for the sake of contradiction, that $x \wedge u = 0$. Find $y' \in C' \subset Y_+$, so that $x \leq y'$. For any $n$, we have $y' \wedge (nu) = (y'-x) \wedge (nu) \leq y'-x$. Thus, $\|y' - y' \wedge (nu)\| \geq \|x\|$. However, $u$ is a quasi-interior point of $Y$, hence $y' = \lim_n y' \wedge (nu)$. 
This is the desired contradiction.

Find $n \in \N$ so that $\|x \wedge u\| > \frac1n$.
Let $I_1,..., I_m$ be the atoms of $\Sigma_n$. For $i \leq m$, define $C'_i = \overline{\ch(\{ Y_m(t):m \geq n, t\in I_i \})}$, and let $C_i = \overline{\so(C'_i)}_+$.

The sequence $(Y_k)$ is a martingale, hence $C' = \overline{\ch(\cup_{i=1}^m C'_i)}$.  Thus,
by Proposition \ref{p:interchange},
$$
C =  \overline{\so(C')} = \overline{\so(\overline{\ch(\cup_{i=1}^m C'_i))}} = \overline{\so(\ch(\cup_{i=1}^m C_i))} .
$$
By \cite[Lemme 3]{BourTal}, $\ch(\cup_{i=1}^m C_i)$ is closed. This set is clearly positive-solid, so by norm continuity of $| \cdot |$, $\so(\ch(\cup_1^m C_i))$ is closed, hence equal to $C$. In particular, $C_+ = \ch(\cup_{i=1}^m C_i)$.
Therefore, if $x$ is an extreme point of $C_+$, then it must belong to $C_i$, for some $i$.
We show this cannot happen. 

If $y \in \so(C'_i)_+$, then we can find $y' \in C_i'$ with $y \leq y'$. 
By parts (1) and (4), $C'_i \subseteq \overline{T(A, f_n(t), \alpha_n(t))}$ for $t\in I_i$, hence, by (3), $\|y' \wedge u\| \leq \frac1n$, which implies $\|y \wedge u\| \leq \frac1n$.
By the triangle inequality, 
$$
\|x \wedge u\| \leq \|y \wedge u\| + \|x-y\| \leq \frac1n + \|x-y\| .
$$
hence $\|x-y\| \geq \|x \wedge u\| - \frac1n$. Recall that $n$ is selected in such a way that $\|x \wedge u\| > \frac1n$. As $C_i = \overline{\so(C'_i)_+}$, it cannot contain $x$.
Thus, $C$ witnesses the failure of the SKMP.
\end{proof}

{\bf Acknowledgments.}
We would like to thank the anonymous referee for reading the paper carefully, and providing numerous helpful suggestions.
We are also grateful to Prof.~Anton Schep for finding an error in an earlier version of Proposition \ref{p:conv_closed_KB}. 

\end{document}